\documentclass[12pt]{article}
\usepackage{geometry,amssymb,amsmath,enumerate, tikz, hyperref, url, ntheorem}
\usepackage{amsfonts}
\usepackage{amssymb}
\usepackage{amsmath}

\newfont{\msbm}{msbm10 scaled\magstephalf}

\newtheorem*{claim}{Claim:}

\newtheorem{theorem}{Theorem}[section]
\newtheorem{corollary}[theorem]{Corollary}

\newtheorem{lemma}[theorem]{Lemma}

\newtheorem{definition}[theorem]{Definition}

\newtheorem{conjecture}[theorem]{Conjecture}

\newtheorem{question}[theorem]{Question}

\def\b1K{\mbox{\boldmath $K$}_{-1}}
\def\bK{\mbox{\boldmath $K$}}

\newbox\noforkbox \newdimen\forklinewidth
\forklinewidth=0.3pt \setbox0\hbox{$\textstyle\smile$}
\setbox1\hbox to \wd0{\hfil\vrule width \forklinewidth depth-2pt
 height 10pt \hfil}
\wd1=0 cm \setbox\noforkbox\hbox{\lower 2pt\box1\lower
2pt\box0\relax}

\def\sub'm{\prec_{\bK'}}
\def\grpf #1 #2{{\rm grp}_{#2}(#1)}
\def\fldf #1 #2{{\rm fld}_{#2}(#1)}
\def\dclf #1 #2{{\rm dcl}_{#2}(#1)}
\def\rclf #1 #2{{\rm rcl}_{#2}(#1)}
\def\aclf #1 #2{{\rm acl}_{#2}(#1)}
\def\acff #1 #2{{\rm acf}_{#2}(#1)}
\def\strf #1 #2{{\rm str}_{#2}(#1)}
\def\tclf #1 #2{{\rm acf}_{#2}(#1)}

\def\hbar{{\bf h}}

\date{\today}

\usepackage{mathtools}

\newtheorem{prob}{Problem}

\usepackage{lineno}

\newenvironment{proof}{\paragraph{Proof:}}{\hfill$\square$\\}
\usepackage{enumitem} 

\begin{document}

\title{ Some Cases of the Erd\H{o}s-Lov\'asz Tihany Conjecture for Claw-free Graphs}
\author{Sean Longbrake\footnote{Dept. of Mathematics, Emory University,  {\tt sean.longbrake@emory.edu}}, Juvaria Tariq\footnote{Dept. of Mathematics, Emory University  {\tt jtariq@emory.edu}}}

\maketitle

\begin{abstract}
    The Erd\H{o}s-Lov\'asz Tihany Conjecture states that any $G$ with chromatic number $\chi(G)  = s + t - 1 > \omega(G)$, with $s,t \geq 2$ can be split into two vertex-disjoint subgraphs of chromatic number $s, t$ respectively. We prove this conjecture for pairs $(s, t)$ if $t \leq s + 2$, whenever $G$ has a $K_s$, and for pairs $(s, t)$ if $t \leq 4 s - 3$, whenever $G$ contains a $K_s$ and is claw-free. We also prove the Erd\H{o}s-Lov\'asz Tihany Conjecture for the pair $(3, 10)$ for claw-free graphs.
\end{abstract}

\section{Introduction}

Throughout this work, $G$ is always a simple graph. We let $K_{\ell}$ denote the complete graph on $\ell$ vertices. We let $K_{s, t}$ denote the complete bipartite graph with one part of size $s$ and one of size $t$.  We let $\chi(G)$ be the least number of colors needed to color the vertices of a graph $G$ such that no edge is monochromatic. We let $\omega(G)$ be the largest $\ell$ such that $K_{\ell} \subset G$. For a set $U \subseteq V(G)$, we let $G[U]$ be the subgraph induced by $U$. We say $G$ is \textit{claw-free} if there is no set $W$ such that $G[W] \cong K_{1, 3}$. For other definitions, see the standard reference \cite{West}. 

We offer now a brief history of the Erd\H{o}s-Lov\'asz Tihany conjecture, with some particularly relevant results highlighted. We direct the reader to \cite{Song} for more details. 

The Erd\H{o}s-Lov\'asz Tihany Conjecture states: 
\begin{conjecture}[Erd\H{o}s-Lov\'asz Tihany \cite{E}]
For $t \geq s \geq 2$, for any graph $G$ with chromatic number $\chi(G)  = s + t - 1 > \omega(G)$ there exists a vertex partition $S \sqcup T = V(G)$ such that $\chi(G[S]) \geq s$ and $\chi(G[T]) \geq t$. 
\end{conjecture}
While this conjecture is quite old and has received much attention over the last fifty years, the exact result is known only for the following pairs: $(2, 2), (2, 3), (2, 4), (3, 3), (3, 4), (3, 5)$ \cite{BJ, M, S1, S2}. 

A particular interesting case is for claw-free graphs. The most general result is the following by Chudnovsky, Fradkin, and Plumettaz \cite{CFP}.  

\begin{theorem}
Let $G$ be a claw-free graph with $\chi(G) > w(G)$. Then, there exists a clique $K$ with $|V(K)| \leq 5$ such that $\chi(G - K) > \chi(G) - |V(K)|$. 
\end{theorem}

In generality, Kostochka and Stiebitz proved the conjecture under the condition $G$ is a line graph \cite{KS}. This was extended to the following by Balogh, Kostochka, Prince, and Stiebitz \cite{BKPS}: 
\begin{theorem}
Any quasi-line graph $G$  with chromatic number $\chi(G)  = s + t - 1 > \omega(G)$ can be split into two disjoint subgraphs of chromatic number $s, t$ respectively. Furthermore, if $\alpha(G) = 2$ and $\chi(G)  = s + t - 1 > \omega(G)$, $G$ can be split into two vertex-disjoint subgraphs of chromatic number $s, t$ respectively.
\end{theorem}

This work was extended by Song \cite{Song2} to the following. 

\begin{theorem}
If $\alpha(G)\geq 3$ and $G$ has no hole of length between $4$ and $2 \alpha (G) - 1$ and $\chi(G)  = s + t - 1 > \omega(G)$, $G$ can be split into two vertex-disjoint subgraphs of chromatic number $s, t$ respectively.
\end{theorem}

As noted by Erd\H{o}s and Lovasz  if $s = 2$, the Erd\H{os}-Lov\'asz Tihany conjecture is equivalent to the following: 

\begin{conjecture} [Double-Critical Graph Conjecture \cite{E}]
If $G$ is a graph such that removing every edge reduces the chromatic number by two, then $G$ is a complete graph. 
\end{conjecture}

This variant has received much attention over the years. In particular, Huang and Yu \cite{HY} proved: 

\begin{theorem}
If $G$ is a claw-free graph of chromatic number six,  such that removing every edge reduces the chromatic number by two, then $G$ is a complete graph
\end{theorem}

Building on this work and work by Kawarabayashi, Pedersen, and Toft \cite{KPT}, Rolek and Song \cite{RS} were able to prove the following: 

\begin{theorem}
If $G$ is a claw-free graph of chromatic number less than or equal to eight,  such that removing every edge reduces the chromatic number by two, then $G$ is a complete graph.
\end{theorem}

In light of this work, we make the following definition

\begin{definition}
For $\ell \geq 2$, we say a graph $G$ is $K_{\ell}$-critical if it satisfies the following three conditions: 
\begin{enumerate}[label=(\roman*)]
\item $G$ has a $K_{\ell}$ as a subgraph. 
\item $G$ is critical, i.e. removing any vertex reduces the chromatic number of $G$ by one. 
\item Removing the vertex set of any $K_{\ell}$ reduces the chromatic number of $G$ by $\ell$. 
\end{enumerate}
\end{definition}

The first two conditions are to remove some trivial examples from the family, such as taking the disjoint union of a $K_{\ell}$-critical graph with chromatic number $k$ with a $K_{\ell}$-free graph of chromatic number $k - \ell$, or taking a $K_{\ell}$-free graph. Note that $K_2$-critical graphs are double-critical graphs.

Note that if $G$ is a counterexample to the Erd\H{o}s-Lov\'asz Tihany Conjecture for a pair $(s, t)$ and contains a $K_s$ as a subgraph, then $G$ contains a $K_s$-critical subgraph of the same chromatic number. Indeed, if $\chi (G - S) > s+ t - 1 - s$ for any copy $S$ of $K_s$, we have found a partition satisfying the Erd\H{o}s-Lov\'asz Tihany Conjecture. As $\chi(G - S) \geq \chi(G) - |V(S)|$, for all subgraphs $S$, we see that if $G$ is a counterexample, then $\chi(G - S) = \chi(G) - s$ for all copies $S$ of $K_s$. In particular, we will prove that any graph with the property that removing any $K_s$ reduces the chromatic number by $s$ contains a $K_s$-critical graph as an induced subgraph.

In light of this, we make the following conjecture. 

\begin{conjecture}\label{klcritical}
If $G$ is $K_{\ell}$-critical for some $\ell \geq 2$, then $G$ is a complete graph. 
\end{conjecture}

Note that while a proof of this conjecture would imply Erd\H{o}s-Lov\'asz Tihany for graphs containing a $K_s$ as a subgraph, the other direction does not hold.

 In \cite{P}, Pedersen offered a similar definition that requires edges to lie on a $K_{\ell}$. In this setting, he proved Conjecture~\ref{klcritical} for $\chi(G) \leq 6$ and $\ell = 3$.  In our work, we drop this requirement that edges lie on a $K_\ell$ and are able to reprove this result, as seen in Corollary~\ref{cor}.

In this language, our main results are the following: 
\begin{theorem}~\label{main1}
If $G$ is a $K_{\ell}$-critical graph with $\chi(G) \leq 2 \ell + 1$, then $G$ is a complete graph. 
\end{theorem}

\begin{theorem}~\label{main2}
If $G$ is a $K_{\ell}$-critical claw-free graph with $\chi(G) \leq 5 \ell - 4$, then $G$ is a complete graph. 
\end{theorem}

Throughout, we call a graph $G$ \textit{triangle-critical} if it is $K_3$-critical. In this case, we can extend the result one step further. 

\begin{theorem}~\label{main3}
If $G$ is a triangle-critical claw-free graph with $\chi(G) = 12$, then $G$ is a complete graph. 
\end{theorem}

We begin with some preliminary lemmas in Section~\ref{lemmas}. Then we will prove Theorem~\ref{main1} in Section~\ref{gen}. Section~\ref{clawfree} will cover the proof of Theorem~\ref{main2}. Finally, the proof of Theorem~\ref{main3} will be in section~\ref{clawfree12}. We will conclude the paper with an open problem.

\section{Preliminary Lemmas}\label{lemmas}
 We define for a set $S \subset V(G)$, $N(S) = \bigcap_{v \in S} N(v)$. For any subgraph $H \subset G$, we define $N(H) = N(V(H))$. We call this set the common neighborhood of $H$. We define $N[S] = N(S) \cup S$. For a subgraph $H$, let the \textit{degree} of $H$,  $d(H) = |N(H)|$.  Furthermore, for any subgraph $F\subseteq G$,  $N_F(H)$ is defined to be $N(H)\cap V(F)$.

\begin{lemma}
Every graph $G$ containing a $K_\ell$ that has the property such that $\chi(G - L) = \chi(G) - |L|$ for every copy $L$ of $K_\ell$ contains a $K_\ell$-critical subgraph $G'$ of the same chromatic number of $G$. 
\end{lemma} 
\begin{proof}
Let $G_0 = G$. Given $G_i$, let $G_{i + 1}$ be formed from $G_i$ by removing some vertex $x \in V(G_i)$ such that $\chi(G_i - x) = \chi(G_i)$. The process stops if no such $x$ remains in $G_i$, and set $G'$ to that graph.

We claim that at every stage of the process, every copy $L$ of $K_\ell$ has the property that $\chi(G_i - L) = \chi(G_i) - \ell$. In particular, this says that at no stage do we remove a vertex $x$ that lies on a $K_{\ell}$. Note that the following holds hold for all copies $L$ of $K_{\ell}$ in $G_i$. 

\begin{align*}
\chi(G - L) & \geq \chi(G_i - L) \geq \chi(G_i) - |L|\\
\chi(G) - \ell &\geq \chi(G_i - L) \geq \chi(G_i) - \ell\\
\chi(G_i) - \ell  &\geq \chi(G_i - L) \geq \chi(G_i) - \ell\\
\end{align*}

Thus, in particular every $G_i$ still has the property that removing a $K_{\ell}$ reduces the chromatic number by $\ell$. Note that by definition, $G'$ is a critical graph. By our earlier arguments, it still has a $K_{\ell}$ and in particular, is thus $K_\ell$-critical.  
\end{proof}

The following two lemmas are equivalent to Lemma $3.1$ and Lemma $3.7$ of Stiebitz \cite{S2}. We include proofs for the sake of completeness. 

\begin{lemma}\label{deg}
Let $G$ be a $K_{\ell}$-critical graph with $\chi(G) = k$. Then, $d(v) \geq k - 1$ for all $v \in V(G)$ and for any $L \subseteq G$, with $L$ a copy of $K_{\ell}$, $d(L) \geq k - \ell$. In particular in any $(k - \ell)$-coloring $\phi$ of $G - L$, for all $i \in [k - \ell]$, $\phi^{-1}(i) \cap N(L) \neq \emptyset$. 
\end{lemma}

\begin{proof}
Since $G$ is critical, for all $v$ in $V(G)$, $\chi( G - v) = k -1$. Fix a coloring of $G - v$ in $k - 1$ colors. If $v$ does not have a neighbor in every color class, then we can color $v$ with the color not used in $N(v)$. This would give a $(k - 1)$-coloring of $G$, contradicting that $\chi(G) = k $. Thus, $v$ sees a neighbor in every color class, and so has degree at least $k - 1$. 

Let $L$ be a $K_{\ell}$ in $G$ and suppose on the contrary that there is a $(k - \ell)$-coloring $\phi$ of $G - L$ where for some $i \in [ k - \ell]$, $N(L) \cap \phi^{-1}(i) = \emptyset$. Fix this $i$.

Let $V(L) = \{ v_{k - \ell + 1}, v_{k - \ell + 2}, \dots v_k \}$. Let $\psi: V(L) \rightarrow [k - \ell + 1, k] : \psi(v_j) = j$. For each vertex $w \in \phi^{-1}(i)$, there is at least one vertex $v_{j_w}$ among $V(L) $ such that $w$ is not adjacent to $v_{j_w}$. Let $f : \phi^{-1}(i) \rightarrow [k - \ell + 1, k]$ such that $f(w) = j_w $.

Define $\phi': V(G) \rightarrow [k] - \{i \}$, a coloring of $G$ as follows

$$\phi'(v) = \begin{cases} \psi(v) & v \in V(L) \\
f(v) & v \in \phi^{-1}(i) \\
\phi(v) & \text{ otherwise} \end{cases} $$  Note that $\phi'$ forms a $(k - 1)$-coloring of $G$, a contradiction to $G$ having chromatic number $k$. 

\end{proof}

\begin{lemma}\label{usefullemma}
Let $G$ be a $K_{\ell}$-critical graph with chromatic number $k$. If $G$ contains $K_{k - \ell + 1}$, then $G \cong K_k$. 
\end{lemma}

\begin{proof}
We will prove this by induction. Note that if $G$ contains $K_k$ the result follows by criticality. Suppose $ 1 \leq i \leq \ell - 1$ and $G$ has a $K_{k - i}$ as a subgraph. Then, if we can show that $G$ contains $K_{k - i+ 1}$, the result would follow. Now, by definition, $G$ has a $K_{\ell}$, so we may assume $k - i \geq \ell$. Let $X = \{x_1, x_2, \dots, x_{k - i}\}$ be the vertices of a $K_{k-i}$.  Suppose $G$ has no $K_{k - i + 1}$.  Note that $G[\{x_1, x_2, \dots, x_{\ell}\}] \cong K_{\ell}$, and let $L = G[\{x_1, x_2, \dots, x_{\ell}\}]$. Also, $\{x_{\ell +1}, \dots x_{k - i}\} \subseteq N(L)$, but since $|N(L)| \geq k - \ell$ by Lemma~\ref{deg}, we have $|N(L) - X| \geq k - \ell - (k - i -\ell) \geq i$.

Since $G$ has no $K_{k - i + 1}$, for every vertex $y$ in $N(L) - X$ there exists an $x \in V(X) - V(L)$ such that $xy$ is not an edge, so without loss of generality fix $y_1 \in N(L) - X$ such that $y_1$ is not adjacent to $x_{\ell + 1}$. Then, $G[\{y_1, x_2, \dots x_{\ell}\}] \cong K_{\ell} = L_1$, so $|N(L_1)|\geq k - \ell$. Note that $x_{\ell + 1}$ is not among the common neighbors of $L_1$, so \begin{align*}
  |N(L_1) - X| &\geq |N(L_1)| - |X - \{x_2, \dots x_{\ell}, x_{\ell + 1}\}|\\
  &\geq k - \ell - (k - i - \ell) \\
  &\geq i. 
\end{align*} Thus, there is a $y_2$ in $N(L_1) - X$. 

Continuing, if $j \leq i$, we have $L_j = G[\{y_1, y_2, \dots y_{j}, x_{j + 1}, \dots x_{\ell}\}] = L_j$ is a copy of  $K_{\ell}$, and we note that $|N(L_j) - X| \geq i + 1 - j$. Then, for all $j \leq i$, there is a $y_{j + 1} \in N(L_j) - X$.  At the end, we have found a $K_{\ell}$, $G[\{ y_1, \dots y_{i + 1}, x_{i + 2}, \dots x_{\ell}\}] = L_{i + 1}$. Let $X' = X - L_{i + 1}$. Note that $|V(X')| = k - i - (\ell - i - 1) = k - \ell + 1$ . Thus, we have found a $K_{\ell}$, namely $L_{i + 1}$, which is vertex-disjoint from a clique $X'$ of size $k - \ell + 1$,  a contradiction to $G$ being $K_{\ell}$-critical . Thus, $G$ has a $K_{k - i +1}$. 
\end{proof}

Note the following immediate corollary: 
\begin{corollary}\label{cor}
    If $G$ is a $K_{\ell}$-critical graph with $\chi(G) \leq 2 \ell$, $G$ is a complete graph. 
\end{corollary}
\begin{proof}
Let $G$ be a $K_{\ell}$-critical graph with $\chi(G) \leq 2 \ell$.  Note that $G$ has a $K_{\ell}$. If $\chi(G) = \ell$, the result is clear. Assume then $\chi(G) > \ell$. In particular, by Lemma~\ref{deg}, we have that the $K_{\ell}$ is contained in a $K_{\ell + 1}$. Thus, by Lemma~\ref{usefullemma}, $G$ is a complete graph. 
\end{proof}

This result is a weaker version of Theorem~\ref{main1}. We will improve it in the next section. 

Following \cite{BKPS, B95, NL}, given a graph $G$ with $k$-coloring $\phi: V(G) \rightarrow [k]$ and a permutation $\pi : [k] \rightarrow [k]$ and a vertex $x \in V(G)$, we let $N_1$ to be the set of vertices adjacent to $x$ with color $\pi(\phi(x))$, $N_2$ the set of vertices adjacent to some vertex in $N_1$ with color $\pi^2(\phi(x))$, $N_3$ the set of vertices adjacent to some vertex in $N_2$ with color $\pi^3(\phi(x))$, and so on. We call $N(x, \phi, \pi) = \{x\} \cup N_1 \cup N_2 \cup \dots$ a \textit{generalized Kempe chain} from $x$ with respect to $\phi$ and $\pi$. Note that changing the color $\phi(y)$ for every $y \in N(x, \phi, \pi)$ to $\pi(\phi(y))$ defines a new $k$-coloring of $G$.

\begin{lemma}\label{kempechain}
Let $G$ be a $K_{\ell}$-critical graph and $L$ be a copy of $K_\ell$in $G$. Let $\chi(G) = k$ and $\phi$ be a $(k -\ell)$-coloring of $G - L$. Then for any nonempty repeat-free sequence $j_1, j_2, \dots j_t$ in $[k - \ell]$, and $x, y \in V(L)$, there is a path on $t + 2$ vertices starting at $x$ and ending at $y$ with the $i+1$th vertex $v$ being in $G - L$ with $\phi(v) = j_i$. 
\end{lemma}

\begin{proof}
    Let $G'$ be the graph on $V(G)$ with edges $E(G) - \{xy\}$. Let $\phi'$ be a $(k-1)$-coloring of $G'$ extending $\phi$ and giving unique colors to every vertex of $L$ besides $x,y$, with $\phi(x) = \phi(y) = k - 1$. Let $\pi$ be the cyclic permutation defined by $(k - 1, j_1, j_2, \dots, j_t)$. If $N(x, \phi, \pi)$ does not contain $y$, then reassigning the colors by applying $\pi$ to the chain (as described above) gives a coloring of $G'$ where $x, y$ have distinct colors. Thus, this would extend to a $k - 1$ coloring of $G$ by adding back the edge $xy$, a contradiction. Therefore, $y$ must be on this generalized Kempe chain. Since only $y,x$ have color $k -1$, it follows that $G[N(x, \phi, \pi)]$ must contain a path from $x$ to $y$ of order $t + 2$ satisfying our conditions. 
\end{proof}

\begin{lemma}\label{avoid}
Let $G$ be a $K_{\ell}$-critical graph with chromatic number $k$ which is not $K_k$. Then there exists a copy $S$ of $K_{\ell + 1}$, such that for every vertex $x \in V(S)$, there is copy $L$ of $K_{\ell}$, satisfying $L \not \subseteq N[x]$. 
\end{lemma}

\begin{proof}
Let $L'$ be a $K_{\ell}$. We will construct $S$ by induction via the following claim. 
\begin{claim}
For any subgraph $S$ contained in a copy $L$ of $K_{\ell}$, there exists $x$ such that $x \in N(S)$ and there is a copy $T$ of $K_{\ell}$, with $T \not \subseteq N[x]$. Moreover, if $|S| \leq \ell - 1$, we can pick $x$ such that $S \cup \{x \}$ is contained in a $K_{\ell}$.
\end{claim}

Note by Lemma~\ref{deg}, that $|N(L)| \geq k - \ell$. Since $G \not \cong K_k$, we have the existence of a pair $x, y \in N(L) \subseteq N(S)$ such that $x \not \sim y$, as otherwise $G[L \cup N(L)] \cong K_k$. Now, $y$ along with $\ell - 1$ vertices of $L$ forms a $K_{\ell}$ not in $N[x]$ as $xy$ is not an edge. If $|S| < \ell$, then we have that $\{x\} \cup S$ with some vertices from $L - S$ forms a $K_{\ell}$. Thus, $x$ is the desired vertex to fulfill the claim.

For our base case, note that $\emptyset \subseteq L'$ satisfies the conditions of the claim. Suppose we have an $S$ satisfying the conditions of claim with $|S| \leq \ell$. Then, by repeatedly applying the above claim, we have that there is an $x$ such that there is a $T$ a copy of $K_{\ell}$, with $T \not \subseteq N[x]$, and $S \cup \{ x\}$ satisfies the claim if $|S| \leq \ell - 1$ and prove Lemma~\ref{avoid} if $|S| = \ell$. 
\end{proof}

\begin{lemma}\label{missneigh}
Let $G$ be a $K_{\ell}$-critical graph with $\chi(G) = k$ and $x$ a vertex in $G$ such that there is a copy $L$ of $K_{\ell}$ with $L \not \subseteq N[x]$. Then, $N(L) \not \subseteq N(x)$. In particular, as $x \not \in N(L)$, this implies  $N(L) \not \subseteq N[x]$.
\end{lemma}

\begin{proof}
  Suppose otherwise and remove $L$ from $G$. Observe that if $x \in V(L)$, then $L \subseteq N[x]$, so $x \in V(G - L)$. Furthermore, $G - L$ is $(k - \ell)$-colorable. Fix a $(k - \ell)$-coloring $\phi$, and note that by Lemma~\ref{deg}, there is a vertex $y \in N(L)$ such that $\phi(y) = \phi(x)$. As $N(L) \subseteq N(x)$, we have a monochromatic edge, contradicting $\phi$ being a coloring. So $N(L) \not \subseteq N(x)$. 
\end{proof}

\begin{lemma}\label{chrom}
Let $G$ be a $K_{\ell}$-critical graph and $x$ a vertex in $G$ such that there is a copy $L_0$ of $ K_{\ell}$ with $L_0 \not \subseteq N[x]$. Then, $\chi(G[N(x)]) \leq k - \ell - 1$. 
\end{lemma}
\begin{proof}

We will need the following claim. 

\begin{claim}
Let $L_i$ be a $K_{\ell}$ intersecting $N(x)$ in $1 \leq s < \ell$ vertices with $x \not \in V(L_i)$. Then there exists a copy $L_{i + 1}$ of $K_\ell$ that intersects $N(x)$ in $s - 1$ places with $x \not \in V(L_{i + 1})$.
\end{claim}

  By Lemma~\ref{missneigh}, there exists a $z \in N(L_{i}) - N(x)$. Since $V(L_{i}) \cap N(x) \neq \emptyset$, there is some vertex $w \in V(L_{i}) \cap N(x) $. Let $L_{i +1} = G[ V(L_i) \cup \{ z\} - \{w\}]$. 

With this claim, we see there is some copy $L_j $ of $ K_{\ell}$ not containing $x$ that intersects $N(x)$ in zero places. Remove $L_j$ from the graph. We have that the remainder is $(k - \ell)$-colorable, so $N(x) \cup \{x\}$ is $(k - \ell)$-colorable. Thus, $N(x)$ is $(k - \ell - 1)$-colorable, as $x$ is adjacent to every vertex within. 
\end{proof}

\begin{lemma}\label{vertexdegree}
Let $G$ be a $K_{\ell}$-critical graph with chromatic number $k$ which is not $K_k$, with $\ell \geq 2$. Then, every vertex which lies on a $K_{\ell}$ has degree at least $k + 2\ell - 3$. In particular, for all $1\leq i \leq \ell$, if $H$ is a copy of $K_i$ is contained in some $K_{\ell}$ in $G$, then $d(H) \geq k - \ell + 3(\ell - i) $. 
\end{lemma}

\begin{proof}
Let $x_1$ be a vertex in $G$ that lies on a $K_{\ell}$. 
Take the $L$ copy of $K_{\ell}$ containing $x_1$ such that over all copies $S$ of $K_{\ell}$ containing $x_1$, $|N(L)| \leq |N(S)|$. Let $V(L)= \{ x_1, \dots, x_{\ell}\}$. By Lemma~\ref{deg}, the number of common neighbors of $V(L)$ is at least $k - \ell$. Since $G$ is not a $K_k$, there is at least one nonedge between two vertices $u, v$ in $N(L)$.

Let $L_i$ denote the $K_{\ell}$ formed by taking $G[\{x_1, \dots, x_{i - 1}, u, x_{i + 1}, \dots, x_{\ell}\}]$, for $i \in [2, \ell]$. Each such $L_i$ has at least as many neighbors as $L$, but does not have $v$ as a neighbor. Thus, $N(L_i ) - ( \{x_i\} \cup N(L)) \neq \emptyset$. Let $z_i$ be in $N(L_i )-  ( \{x_i\} \cup N(L)) $. Since $z_i \not \in N(L)$, but $z_i \in N(\{x_1, \dots x_{i -1}, x_{i + 1}, \dots x_{\ell}\})$, we know that $z_i \not \in N(x_i)$. Thus by Lemma~\ref{missneigh},  we have that $L_{i}' = G[\{x_1, \dots, x_{i - 1}, z_i, x_{i + 1}, \dots, x_{\ell}\}]$ satisfies $N(L_i') \not \subseteq N[x_i]$. Let $z_i'$ be a vertex in $N(L_i') - N[x_i]$. Note in particular, $z_i' \not \in N(L)$. 

Note that $z_i \neq z_j$ for $i \neq j$, as then $z_i$ would be in $N(L)$. Furthermore, we have that $z_i' \neq z_j$, as $z_i'$ is not adjacent to $x_i$, yet $z_j$ is adjacent to $x_i$. Similarly, $z_i' \neq z_j'$ for $i \neq j$.

Thus, $d(x_1) \geq |V(L) - \{x_1\}| + d(L) + |\{ z_2, z_2', \dots z_{\ell}, z_{\ell}'\}| \geq \ell - 1 + k - \ell + 2(\ell - 1) \geq k + 2\ell - 3$. Via the previous argument, any $K_i$ contained in a $K_{\ell}$ has at least $k - \ell + 3(\ell - i)$ many common neighbors. 
\end{proof}




\begin{lemma}\label{outside}
    Let $G$ be a $K_\ell$-critical graph. Let $x$ be any vertex of $G$ and $v$ be a vertex lying on a copy $L_0$ of $K_{\ell}$ which contains a vertex outside $N[x]$. Then $v$ has at least $\ell$ neighbors outside of $N[x]$. 
\end{lemma}

\begin{proof}

We will use the following claim. 

\begin{claim}
Let $L_i$ be a $K_{\ell}$ containing a vertex $v$ but not $x$ such that $|( L_i - \{v\} )\cap N(x) | = s$ with $1 \leq s < \ell - 1$. Then there exists $L_{i + 1}$ such that $| (L_{i + 1} - \{v\}) \cap N(x) | = s - 1$, $x \not \in V(L_{i + 1})$,  $L_{i + 1} \cong K_{\ell}$, and $v$ still lies on $L_{i + 1}$.
\end{claim}

  By Lemma~\ref{missneigh}, there exists a $z \in N(L_{i}) - N[x]$. Since $V(L_{i}) \cap N(x) -\{v \} \neq \emptyset$, there is some vertex $w \in V(L_{i}) \cap N(x) -\{v\} $. Let $L_{i +1} = G[ V(L_i) \cup \{ z\} - \{w\}]$.

By this claim, there is some copy $L_j$ of $K_{\ell}$ such that $(L_j - \{v \}) \cap N[x] = \emptyset$. By Lemma~\ref{missneigh}, there is a $z \in N(L_j) - N[x]$. Thus, there are at least $\ell$ vertices in $N(v) - N[x]$, namely $V(L_j) - \{v\}$ and $z$.

\end{proof}

\section{Proof of Theorem~\ref{main1}}\label{gen}

Note that Theorem~\ref{main1} follows from below and Corollary~\ref{cor}. 
\begin{theorem}
Let $G$ be a $K_{\ell}$-critical graph with chromatic number $2 \ell + 1$ with $\ell \geq 2$. Then, $G \cong K_{2\ell + 1}$. 
\end{theorem}

\begin{proof}
Assume otherwise, and let $G$ be such a graph. Fix a copy $X$ of $K_{\ell}$, and let $\{ x_1, x_2 \dots x_\ell\} = V(X)$ inside $G$, and fix a $(\ell + 1)$-coloring of $G - X$, $\phi: V(G - X) \rightarrow [ \ell + 1]$. Let $a_1$ be one of the at least $\ell + 1$ common neighbors of $X$. For $i < \ell$, having defined $a_1, \dots a_i$, let $a_{i + 1}$ be a common neighbor of $\{a_1, \dots a_i, x_{i + 1} \dots x_{\ell}\}$ which is not among $x_1, \dots x_{i}$. As by Lemma~\ref{deg} the common neighborhood has size $\ell + 1$, we have that there is such a choice.

Now, for $i \geq \ell$, having defined $a_{i - \ell + 1}, a_{i - \ell + 2}, \dots a_i$, we define $a_{i + 1}$ as any common neighbor of these vertices among $V(G - X)$ yet to appear on our sequence. We stop when no choices remain.

Note that this sequence $\{a_1, a_2 \dots, a_p\}$ is uniquely $(\ell + 1)$-colorable by construction, i.e. if $\phi$ and $\phi'$ are two $(\ell + 1)$-colorings of $\{a_1, a_2, \dots a_p\}$, there is a permutation $\pi: [\ell + 1] \rightarrow [\ell + 1]$ such that $\pi \circ \phi = \phi'$. Since the sequence induces a subgraph of $V(G - X)$, it is $(\ell + 1)$-colorable, and as it can be seen as a sequence of $K_{\ell + 1}$'s intersecting in $K_{\ell}$'s, there is a unique up to relabeling way to do it: coloring $a_i$ with $i \pmod{\ell + 1}$, making it uniquely $(\ell + 1)$-colorable. Let $L=G[\{a_{p-\ell + 1}, a_{p - \ell + 2}, \dots, a_{p}\}]$ be the last $K_{\ell}$ on the sequence. Lemma~\ref{usefullemma} implies that $G$ is $K_{\ell + 2}$-free, hence { $d_X(L)\leq 1$}. Furthermore, by Lemma~\ref{deg}, $d_{G-X}(L) \geq \ell$. Given that $L$ is a $K_{\ell}$, in any $(\ell + 1)$-coloring of $G-X$, $N_{G-X}(L)$ is monochromatic. As $L$ is the last $K_{\ell}$, $N_{G - X}(L) \supseteq \{a_{b_1}, a_{b_2}, \dots a_{b_{\ell}}\}$ lies in the sequence. Thus, $b_1, b_2, \dots b_{\ell} \equiv p + 1 \pmod{\ell + 1}$. So, letting $b_1$ be smallest among $\{b_1, b_2, \dots b_\ell\}$, we have that $p + 1 - \ell (\ell + 1) \geq b_1 \geq 1$, therefore $p \geq \ell ( \ell + 1)$.

 Note that since $p \geq \ell (\ell + 1)$, $L$ is distinct from $a_1, a_2, \dots, a_{\ell}$. By our earlier observation, there is a $j$ such that $N_{G - X}(L) \subseteq \phi^{ -1}(j)$.
Let $A = \phi^{-1}(j) \cap \{a_1, a_2, \dots a_{p - \ell }\}$. Now, by unique colorability, any $(\ell + 1)$-coloring of $G - L$ colors $A$ with one color. As $L$ has at most one common neighbor among $X$, we have that the neighborhood of $L$ sees at most two color classes of the coloring of $G - L$, and thus misses at least one. But this contradicts Lemma~\ref{deg}, thus $G \cong K_{2 \ell + 1}$. 
\end{proof}

\section{Proof of Theorem~\ref{main2}}\label{clawfree}

We will now show the following. We include this proof to highlight the alternative method using Ramsey numbers and their lower bound constructions. 
\begin{theorem}
If $G$ is triangle-critical, has chromatic number eight, and is claw-free, then $G \cong K_{8}$.
\end{theorem} 

\begin{proof}
Let $G$ be a triangle-critical claw-free graph with chromatic number eight. Then, $d(T) \geq 5$ for all triangles $T$ by Lemma~\ref{deg}.

Fix any triangle $T$. We will now show that $G[N(T)] \cong C_5$ for $T$. If $G$ is not $K_8$, then the neighborhood of $T$ is $K_3$-free by Lemma~\ref{usefullemma}. Furthermore, since $G$ is claw-free, the independence number of $G[N(T)]$ is at most two. Suppose on the contrary that $d(T) \geq 6$. Then since $R(3, 3) = 6$, $N(T)$ contains either a triangle or an independent set of size three, a contradiction. Thus, $d(T) = 5$. Consequently, from the uniqueness of the lower bound Ramsey construction, $G[N(T)] \cong C_5$. 

Let $T_1$ be a triangle in $G$, with $V(T_1) = \{x, y, z\}$. Let $N(T_1) = \{a, b, c, d, e\}$, which forms a cycle $(a , b , c , d , e )$. Now, take the triangle $T_a$ induced by $\{x, y, a\}$, the common neighborhood of this triangle certainly contains the vertices $z, b$, and $e$. As before, $G[N(T_a)] \cong C_5$, so there must be vertices, which we will call suggestively $a_1, e_1$ lying in the common neighborhood such that $(z , b , a_1 , e_1 , e )$ is a cycle. Note that  $a \not \sim c, d$, therefore $a_1, e_1 \neq c, d$. 

Now, let us examine the triangle $T_b$ induced by $\{x, y, b\}$. $N(T_b)$ includes $a, a_1, z, c$. We already know that $c \sim z \sim a \sim a_1$, so there must be a $b_1 \in N(T_b)$ such that $c \sim b_1 \sim a_1$. As  $b_1 \sim b$, we have that $b_1 \neq e_1$ since $e_1 \not \sim b$. Similarly, $b_1 \neq d, e$. 

Let us now look at the triangle $T_c$ induced by $\{x, y, c\}$. We note that $\{b, d, b_1, z\} \subseteq N(T_c)$. There must be a fifth vertex $c_1$ such that $c_1 \sim b_1, d$ but $c_1 \not \sim b, z$. Thus, $c_1 \neq a_1, b_1, a, e$, yet it may be true that $c_1 = e_1$. We will discount this possibility later. 

Consider now the triangle $T_d$ induced by $\{x, y, d\}$. We note that $\{c, e, c_1, z\} \subseteq N(T_d)$. There must be a fifth vertex $d_1$ such that $d_1 \sim c_1, e$ but $d_1 \not \sim z, c$. Note that for $G[N(T_d)] \cong C_5$, we must have that $c_1 \not \sim e$. Therefore, $c_1 \neq e_1$. Since $d_1 \not \sim c$, we have that $d_1 \neq b_1, b$. 

Let us look at the triangle $T_e$ induced by $\{x, y, e\}$. We note that $\{a, d, z, d_1, e_1 \} \subseteq N(T_e)$. If $e_1 = d_1$, then a subset of $N(T_e)$ would induce a $C_4$, so $e_1 \neq d_1$. Thus, in particular $d_1 \not \sim a$, so $d_1 \neq a_1$. Therefore all five $a_1, b_1, c_1, d_1, e_1$ are distinct and distinct from $\{a, b, c, d, e \}$.

Now, let us take the triangle $T_{a_1}$ induced by $\{x, y, a_1\}$. We note that $N(T_{a_1})$ contains the vertices $a, b, b_1, e_1$ and a vertex $w$ such that $(e_1, a , b , b_1 , w)$ is a $C_5$. Then, we look at the triangle $T_{b_1}$ induced by $\{x, y, b_1\}$, $N(T_{b_1})$ contains $b, c, a_1, c_1, w$. If $w = c_1$, $G[N(T_{b_1})]$ would contain a $C_4$, so we have that $w \neq c_1$ and $w \sim c_1$. Via similar arguments examining $N(\{x, y, c_1\})$, $N(\{x, y, d_1\})$, $N(\{x, y, e_1\})$, we have that $N(\{x, y, w\}) = \{a_1, b_1, c_1, d_1, e_1\}$. Note in particular that $w$ is distinct from all five of these vertices. By triangle-criticality, $G -\{a, b, a_1\}$ has chromatic number five.  Fix a coloring. Note that $x,y$ must have distinct colors from $\{w, z, c, d, e, b_1, c_1, d_1, e_1\}$, so we must color the rest with three colors. $\{b_1, c_1, c\}$ must all receive three distinct colors, say respectively 1, 2, 3. $\{b_1, c_1, w\}$ is a triangle, so $w$ must receive color $3$. $\{c, c_1, d\}$ is a triangle so $d$ receives color $1$. $\{d, c_1, d_1\}$ is a triangle, so $d_1$ receives color $3$. Yet $d_1 \sim w$, a contradiction.

Thus, $G \cong K_8$. 
\end{proof}

Furthermore, we will prove the following statement:

\begin{theorem}[Restatement of Theorem~\ref{main2}]
Let $\ell \geq 2$. Let $G$ be $K_{\ell}$-critical claw-free graph with chromatic number $k \leq 5 \ell - 4$. Then $G \cong K_k$.
\end{theorem}

\begin{proof}
Suppose on the contrary, $G \not \cong K_{k}$. 

Then, by Lemma~\ref{avoid}, there is a copy $S$ of $K_{\ell}$ such that for every $x \in V(S)$, there is some copy $L$ of $K_{\ell}$ such that $L \not \subseteq N[x]$. 

Moreover, for any $x \in V(S)$, Lemma~\ref{chrom} implies $\chi(G[N(x)]) \leq k - \ell - 1$, and hence by claw-freeness $d(x) \leq 2 (k - \ell - 1)$.

Let $u, v$ be a nonadjacent pair inside $N(S)$. Note that $G$ is a $K_k$ if no such pair exists. Then for every $x \in V(S)$, by Lemma~\ref{vertexdegree}, $d(u, x) \geq k - \ell + 3(\ell - 2) \geq k + 2 \ell - 6$. 

Take $y \in V(S)$, let $S' = G[V(S) - \{y\} \cup \{v\}]$. Then, $S'$ is not in the neighborhood of $u$ and does not contain $u$, but contains every $x \in S - \{y\}$.  Thus, by Lemma~\ref{outside}, every $x \in V(S) - \{y\}$ has at least $\ell$ neighbors outside $N[u]$. Thus, since $x$ is adjacent to $u$, we have that 
\begin{align*}
    d(x) &\geq d(x, u) + |N(x) - N[u]| + |\{u\}|\\
    &\geq k + 2 \ell - 6 + \ell + 1\\
    &\geq k + 3 \ell - 5.
\end{align*}
Combining this with the upper bound on $d(x)$, we have

\begin{align*}
    k + 3 \ell - 5 &\leq 2(k - \ell - 1)\\
    k + 3 \ell - 5 & \leq 2k  - 2 \ell - 2\\
    5 \ell - 3 &\leq k.
\end{align*}

Yet, by assumption, we have that $k \leq 5 \ell - 4$, a contradiction. So $G \cong K_k$. 
\end{proof}

\section{Proof of Theorem~\ref{main3}}\label{clawfree12}

We will now prove the following statement: 
\begin{theorem}[Restatement of Theorem~\ref{main3}]
Let $G$ be a triangle-critical, claw-free graph of chromatic number twelve. Then $G \cong K_{12}$
\end{theorem}

\begin{proof}

    Assume on the contrary that $G \not \cong K_{12}$. Let $a$ be a vertex of $G$ that lies on a triangle $L$ such that there is a triangle $L'$ in $G -\{a\}$ not fully contained in $N(a)$. By Lemma~\ref{avoid}, such a vertex exists. Let $b,d$ be a nonedge in $N(L)$. Since $G \not \cong K_{12}$, such a nonedge exists. Note that there is a triangle containing $a$, which does not lie inside $N(b)$ and does not contain $b$. Recall that by claw-freeness and Lemma~\ref{chrom}, $d(a), d(b)\leq 2(12 - 3 - 1) \leq 16$.
    
    If every triangle containing $ab$ has degree at least ten,  then, following the proof of Lemma~\ref{vertexdegree}, we have that $d(a, b) \geq 13$. As by Lemma~\ref{outside}, $|N(a) - N[b]| \geq 3$, we have that $$d(a) = d(a, b) + |N(a) - N[b]| + |\{b\}| \geq 17,$$ a contradiction. Thus, by Lemma~\ref{deg}, there is a $c \in N(a, b)$ such that $d(a, b, c) = 9$. Let $T = G[\{a, b, c\}]$. Now, by Lemma~\ref{vertexdegree}, we have that $d(a, b), d(a, c), d(b, c) \geq  12$. Thus, there are six vertices, $x, x', y, y', z, z'$ such that $x, x' \in N(a, b)  - N[c]$, $y, y' \in N(a, c) - N[b]$, and $z, z' \in N(b, c) - N[a]$. 

    Now, we have that for each vertex among $\{a, b, c\}$ there is a triangle not containing it such that misses its neighborhood, so by Lemma~\ref{chrom} and claw-freeness, $d(a), d(b) ,d(c) \leq 16$. 

    Let us examine the triangle $S = G[\{a, b, x\}]$. Since $x \not \in N[c]$, $S$ has a neighbor outside $N[c]$ by Lemma~\ref{missneigh}; without loss of generality, we may assume it is $x'$. Also by Lemma~\ref{missneigh}, $a, x, x'$ has a common neighbor outside the $N[c]$, let us call it $a'$. Suppose $a' \sim b$. Then, $d(a,b) \geq 13$. But, then by Lemma~\ref{outside}, $d(a) \geq 17$, a contradiction. So $a' \not \sim b$. Similar logic gives a vertex $b'$ that is adjacent to $x, x', b$ but not $a, c$.

    Now, let us examine $S' = G[\{a, c, y\}]$. Since $y \not \in N[b]$, $S'$ has a neighbor outside $N[b]$ by Lemma~\ref{missneigh}; without loss of generality, we may assume it is $y'$. Furthermore, $a, y, y'$ has a common neighbor outside $N[b]$, let us call it $a''$. As above $a'' \not \sim c$. If $a'' \neq a'$, then $d(a) \geq 17$, a contradiction. Thus, $a'' = a'$.

    Following this logic to its natural conclusion, we have found that $x \sim x'$, $y \sim y'$, and $z \sim z'$, and the existence of three vertices $a', b', c'$ such that $a' \sim a, x, x', y, y'$; $a' \not \sim b, c$; $b' \sim b, x, x', z, z'$; $b' \not \sim a, c$; $c' \sim c, y, y', z, z'$; and $c' \not \sim a, b$.

    Note that the edge $aa'$ lies on a triangle, so $d(a, a') \geq 12$. In particular, $N(a, a')$ contains $x, x', y, y'$ and eight vertices among $N(T)$. Similar logic holds for $bb'$ and $cc'$. In particular $N_{N(T)}(a', b', c') \geq 6$. Fix $w \in N_{N(T)}(a', b', c')$. If $a' \not \sim b'$, then $G[\{w, a', b', c\}]$, would be a claw, a contradiction. So $a' \sim b'$, and similar logic shows  $T' := G[\{a', b', c'\}]$ satisfies $T' \cong K_3$, as shown in Figure 1.

\begin{figure}[h!]\centering\label{fig1}
    \includegraphics[scale=0.80, page=3]{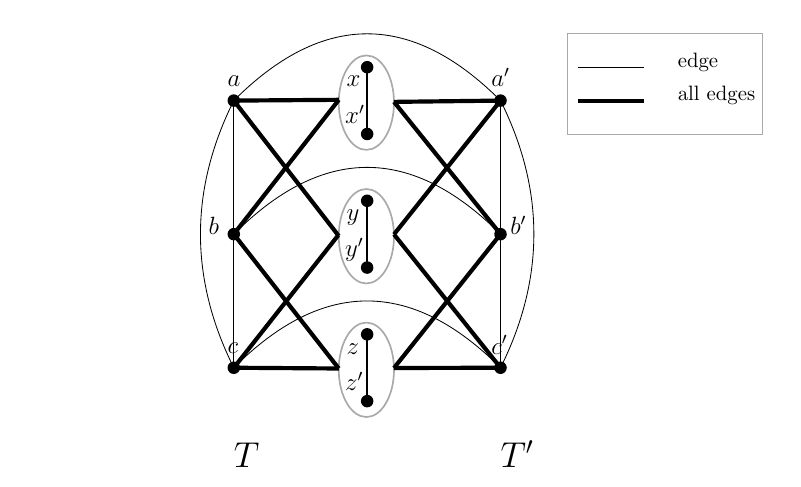}
    \caption{Some edges in $G[\{x, x', y, y', z, z'\} \cup T\cup T']$}
\end{figure}

    Now, $G - T$ is $9$-chromatic by triangle-criticality. Fix one such $9$-coloring $\phi$. Under $\phi$, $N(T)$ receives all nine colors by Lemma~\ref{deg}. Since $d_{N(T)}(a') \geq 8$, we have that $a'$ has exactly one non-neighbor in $N(T)$. Suppose inside $N(T)$, $a', b'$ share a common non-neighbor. Let $v$ be the non-neighbor of $a'$ in $N(T)$, and so under $\phi$, $\phi(v) = \phi(a')$. Under the assumption $a', b'$ are both nonadjacent to $v$, we have that $\phi(v) = \phi(b')$. But then, $\phi(b')= \phi(a')$, contradicting that $\phi$ is proper coloring. 

    This gives us a complete description of the connectivity between $N(T)$ and $T'$, as shown in Figure 2. 

\begin{figure}[h!]\centering\label{fig2}
    \includegraphics[scale=0.65, page=1]{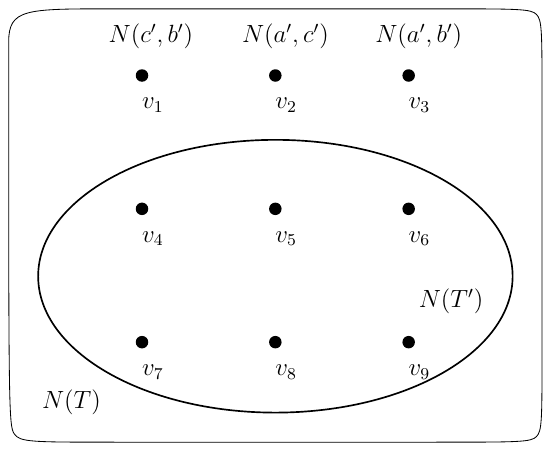}\caption{Structure of $N(T)$}
\end{figure}

    Let us now examine the edge $ax$. As it lies on a triangle $d(ax) \geq 12$. Since $ax$ has at most three neighbors among $\{x', y, y', z, z'\}$ and exactly two neighbors, $a', b$ among $\{b,c, a', b', c'\}$, we have that $ax$ have at least seven common neighbors among $N(T)$. Similarly, $d_{N(T)}(x'), d_{N(T)}(y), d_{N(T)}(y'), d_{N(T)}(z), d_{N(T)}(z') \geq 7$. 

    Note that under $\phi$, six colors appear among $x, x', y, y', z, z'$. Indeed, suppose that two of them received the same color, say without loss of generality $\phi(x) = \phi(y)$. Letting $w$ be the vertex in $N(T)$ receiving color $\phi(x)$, then $G[\{a, x, y, w\}]$ would be a claw, a contradiction.

    We seek to show that, under $\phi$, every vertex in $\{x, x', y, y', z, z', a', b', c'\}$ receives a distinct color.  Suppose on the contrary that one vertex among $x, x', y, y', z, z'$ under the coloring $\phi$ shares a color with one of $\{a', b', c'\}$. Without loss of generality, we may assume it is $z$ and $a'$. Let $v_i$ be the vertex in $N(T)$ such that $\phi(v_i) = \phi(z) = \phi(a')$. Now, as $z, a' \not \sim v_i$, we have that $d_{N(T)}(z, a') \geq 7$. If $d_{N(T)}(v_i) \geq 2$, then $N_{N(T)}(v_i, a', z) \neq \emptyset$, and so $G$ would a contain a claw, a contradiction. Thus,  $d_{N(T)}(v_i) \leq 1$. Now, $d(a, b, v_i) \geq 9$, so $|N(a) - N(c)| \geq d(a, b, v_i) - d_{N(T)}(v_i) \geq 8$. But, then $d(a) = d(a, c) + |N(a) - N(c)| \geq 12 + 8 > 16$, a contradiction. Thus, under $\phi$, every vertex in $\{x, x', y, y', z, z', a', b', c'\}$ receives a distinct color.

    Without loss of generality, assume $\phi(a') = 1, \phi(b') = 2, \phi(c') = 3, \phi(x) = 4, \phi(x') = 5, \phi(y) = 6, \phi(y') = 7, \phi(z) = 8, \phi(z') = 9$. For every $i \in [9]$, let $v_i \in N(T)$ be the unique vertex colored $i$ under $\phi$. 

\begin{claim}
    For all $i \in \{1, 2, 3\}$ and all $j \in \{4, 5, 6, 7, 8, 9\}$, $v_iv_j$ is an edge.
\end{claim}     For simplicity, let us first examine $i = 1, j = 4$. By Lemma~\ref{kempechain}, there is a path of order four from $b$ to $c$ where the second vertex receives color $1$ and the third vertex receives color $4$. Yet $b$ is adjacent to exactly one vertex of color $1$, $v_1$, and $c$ is adjacent to exactly one vertex of color $4$, $v_4$. So $v_1 v_4$ is an edge. Similar arguments complete this claim.

    Our final claim before our contradiction is that $d(T')= 9$. Now, as $T$ is a triangle that does not lie completely in any of their neighborhoods, by Lemma~\ref{chrom} and claw-free, we have that $d(a') , d(b') , d(c') \leq 16$. Suppose on the contrary that $d(T') \geq 10$. Now, $a'$ is adjacent to at least seven vertices that are not in $N(T')$, namely $a, v_2, v_3, b', c'$, and then at least two of $x, x', y, y'$, as $c'$ cannot be adjacent to both of $x, x'$ and $b'$ cannot be adjacent to both of $y, y'$, as then either $d(c') \geq 17$ or $d(b') \geq 17$ respectively. Indeed, if $c'$ were adjacent to both $x, x'$, then $c'$ would be adjacent to eight vertices in $N(T)$, $c, x, x', y, y', z, z'$ and $a', b'$. Yet then, $d(a') \geq 17$, so $d(T') = 9$.  
    
\begin{figure}\label{fig3}
    \includegraphics[scale=0.8, page=2]{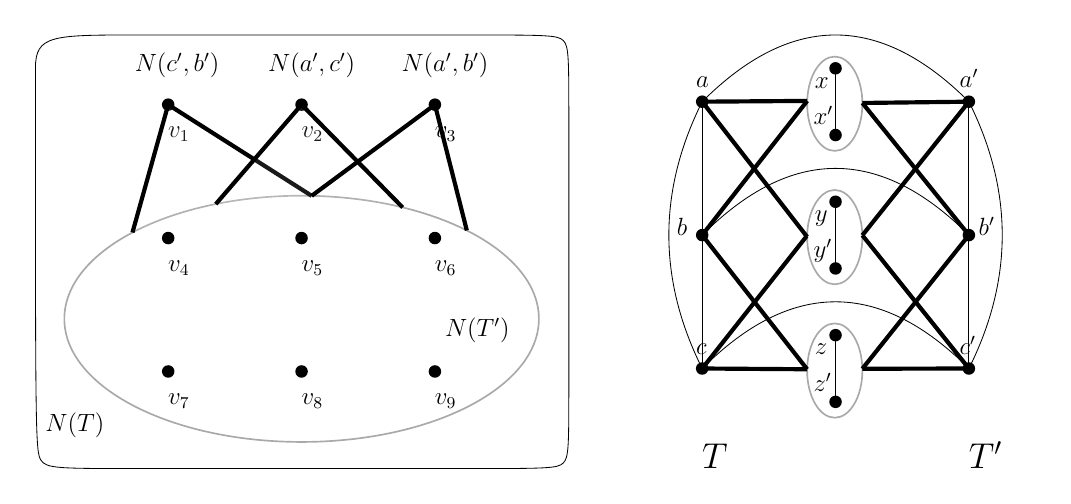}
    \caption{Key Structure of $N(a) \cup N(b) \cup N(c)$}
\end{figure}

We will now show that $\chi(G - T') \geq 10$, contradicting triangle-criticality. Suppose there is a $9$-coloring of $G - T'$, call it $\psi$. Then, by Lemma~\ref{deg}, all nine colors must appear in the $N(T')$. As there are only nine vertices, every vertex must get a distinct color. Suppose without loss of generality that $\psi(v_4) = 4, \psi(v_5) = 5, \psi(v_6) = 6, \psi(v_7) = 7, \psi(v_8) = 8, \psi(v_9) = 9$. Then, as each of $\{a, b, c\}$ is adjacent to all six of these vertices, and form a triangle, we may assume $\psi(a) = 1, \psi(b) = 2$, and $\psi(c) = 3$. Thus, all nine colors appear in the neighborhood of $v_1$, and so $\psi$ cannot be a proper coloring.

Therefore, $G \cong K_{12}$.

\end{proof}

\section{Concluding Remarks}

We note that if $G$ is a counterexample for the Erd\H{o}s-Lov\'asz Tihany conjecture for a pair $(s, t)$ with $s = 3$, then it must have a $K_3$. In particular, Stiebitz [Lemma $3.6$,  \cite{S2}] showed it must have a $K_4$. Thus, our main results reprove Erd\H{o}s-Lov\'asz Tihany for $(3, 3), (3, 4), (3, 5)$ and prove it for claw-free graphs for $(3, t)$ with $t \in \{6, 7, 8, 9, 10\}$. In particular, this leads to the open question: 
\begin{question}
Does any counterexample $G$ to the Erd\H{o}s-Lov\'asz Tihany Conjecture for a pair $(s, t)$ with $s, t \geq 4$ require $K_s$ to a be a subgraph of $G$?
\end{question}

\section*{Acknowledgements}

The authors would like to thank Alexander Clifton for valuable discussions regarding the presentation and content of this work.

\bibliography{bibliography}
\bibliographystyle{alpha}

\end{document}